\documentclass[a4paper]{amsart}

\usepackage{amssymb}
\usepackage{amsthm}  
\usepackage{amsmath} 
\usepackage{amscd} 
\usepackage[all]{xypic}
\usepackage{url}
\usepackage{multirow}

\newtheorem{thm}{Theorem}
\newtheorem{prop}{Proposition}
\newtheorem*{prop*}{Proposition}
\newtheorem*{thm*}{Theorem}
\newtheorem{lem}{Lemma}
\newtheorem*{lemma*}{Lemma}
\newtheorem{cor}{Corollary}
\newtheorem*{cor*}{Corollary}

\newtheorem*{rem*}{Remark}
\theoremstyle{definition}
\newtheorem*{def*}{Definition}

\theoremstyle{remark}
\newtheorem*{exam}{Example}

\begin{document}
\title{Classification of invariant AHS--structures on semisimple locally symmetric spaces}
\author{Jan Gregorovi\v c}
\address{
Jan Gregorovi\v c, Department of Mathematics, Faculty of Sciences and Technology, Aarhus University, Ny Munkegade 118, 8000 Aarhus C, Denmark
}
\email{jan.gregorovic@seznam.cz}
\begin{abstract} 
In this article, we discuss which semisimple locally symmetric spaces admit an AHS--structure invariant to local symmetries. We classify them for all types of AHS--structures and determine possible equivalence classes of such AHS--structures. 

\end{abstract}
\maketitle

\section{Introduction}\label{1.1}

One can describe (locally) symmetric spaces employing the language of Cartan geometries. This allows us to use a functorial construction of invariant geometric structures. The construction is called extension and its application to symmetric spaces is described in \cite{greg}. We will review basic facts of the construction and then apply this construction in the case of almost hermitian symmetric (shortly AHS-) structures. This will allow us to obtain full classification of these structures on semisimple locally symmetric spaces. The AHS--structures are large class of geometric structures, which can be described as one graded parabolic geometries and cover the cases of projective, conformal and many other interesting geometric structures. We will use the classification and description of AHS--structures as it is summarized in \cite{parabook}.

\subsection{Locally symmetric spaces}

Firstly, we review a description of locally symmetric spaces using the language of Cartan geometries. This will use the fact that the Cartan geometries preserves many properties of the model geometries, which are in this case symmetric spaces. The description of symmetric spaces using Cartan geometries is summarized in \cite{greg}. This allows us to use the following definition of locally symmetric spaces.

\begin{def*}
A locally symmetric space is a locally flat Cartan geometry $(p: \mathcal{G}\to M, \omega)$ of type $(K,H,h)$, where
\begin{itemize}
\item $M$ is a smooth, connected manifold
\item $K$ is a Lie group with Lie subgroup $H$, $K/H$ connected, $h\in H$ such that $h^2=id_K$ and $H$ is open in the centralizer of $h$ in $K$
\item the maximal normal subgroup of $K$ contained in $H$ is trivial
\item $\mathcal{G}$ is a principal $H$-bundle over $M$
\item $\omega$ is a Cartan connection, i.e. a $\mathfrak{k}$-valued one form on $\mathcal{G}$, which is $H$-equivariant, reproduces the fundamental vector fields and provides isomorphism $T\mathcal{G}=\mathcal{G}\times \mathfrak{k}$.
\end{itemize}

We say that a locally symmetric space of type $(K,H,h)$ is semisimple if $K$ is semisimple, and we say that the homogeneous model $(K\to K/H,\omega_K)$ is a symmetric space, where $\omega_K$ is the Maurer-Cartan form.
\end{def*}

Since the Cartan geometry is locally flat, the geometry is locally isomorphic to an open subset of the homogeneous model $(K\to K/H,\omega_K)$. In particular, there is an atlas of $M$ with images in $K/H$ such, that the transition maps are restriction of left multiplication of elements of $K$. This allows us to locally define local symmetries by setting $S_{kH}lH:=khk^{-1}lH$. These are local automorphisms of the Cartan geometry, which obviously does not depend on the chosen map from the atlas.

Consequently, one can classify up to local equivalence the locally symmetric spaces in the following manner:

\begin{prop}\label{p1}
The locally symmetric space of the type $(K,H,h)$ is up to local equivalence determined by the pair $(\mathfrak{k},\mathfrak{h})$.
\end{prop}
\begin{proof}
This follows from the fact that $K\to K/H$ is locally equivalent to its simply connected covering.
\end{proof}

We note that in \cite{bert}, the equivalence class corresponding to pair $(\mathfrak{k},\mathfrak{h})$ is  called a germ of symmetric spaces. As we will see later, for our area of interest, the exact type $(K,H,h)$ plays a important role, but for many question, we will restrict to the pair $(\mathfrak{k},\mathfrak{h})$.

The classification of pairs $(\mathfrak{k},\mathfrak{h})$, which are also called symmetric Lie algebras, was in the semisimple case presented in \cite{berg} and here we recall the main classification theorem:

\begin{prop}
Each semisimple symmetric Lie algebra is a finite sum of symmetric Lie algebras of the following two types (which are called simple):

\begin{itemize}
\item $(\mathfrak{h}\oplus \mathfrak{h},\mathfrak{h})$ (we denote them simply by $\mathfrak{h}$), where $\mathfrak{h}$ is a simple Lie algebra and the second $\mathfrak{h}$ is the diagonal in $\mathfrak{h}\oplus \mathfrak{h}$.

\item$(\mathfrak{k},\mathfrak{h})$, where $\mathfrak{k}$ is a simple Lie algebra. The pairs $(\mathfrak{k},\mathfrak{h})$ of all possible cases can be found in the table in \cite{berg}.
\end{itemize}
\end{prop}

\subsection{Extension of Cartan geometries}

Here, we summarize the functorial construction called extension, which will allow us to construct every Cartan geometry on locally symmetric space invariant to local symmetries.

\begin{def*}
Let $P$ be a Lie subgroup of a Lie group $G$ and let $H$ be a Lie subgroup of a Lie group $K$. We say, that pair a $(i,\alpha)$ is an extension of $(K,H)$ to $(G,P)$ if it satisfies:
\begin{itemize}
\item $i:H\to P$ is a Lie group homomorphism
\item $\alpha: \mathfrak{k}\to \mathfrak{g}$ is a linear mapping extending $T_ei:\mathfrak{h}\to \mathfrak{p}$
\item $\alpha$ induces a vector space isomorphism of $\mathfrak{k}/\mathfrak{h}$ and $\mathfrak{g}/\mathfrak{p}$
\item $Ad(i(h))\circ \alpha=\alpha \circ Ad(h)$ for all $h\in H$ i.e. $\alpha$ is a homomorphism of the representations $Ad(H)$ and $Ad(i(H))|_{\operatorname{Im}(\alpha)}$
\end{itemize}
\end{def*}

Having an extension $(i,\alpha)$ of Cartan geometry $(p: \mathcal{G}\to M, \omega)$ of type $(K,H)$, one forms bundle $\mathcal{G}\times_iP\to M$ and defines one form $\omega_\alpha$ by $\omega_\alpha|_{\mathcal{G}\times_ii(H)}=\alpha\circ \omega$ and $\omega_\alpha|_{i^{-1}(P)\times_iP}=\omega_P$ and extending it by right $P$ action to be equivariant. The general theory of \cite{parabook} and results in \cite{greg} allows us to formulate the following theorem:

\begin{thm}\label{t1}
Let $(i,\alpha)$ be an extension of $(K,H)$ to $(G,P)$ and $(p: \mathcal{G}\to M, \omega)$ a Cartan geometry of type $(K,H)$. Then the pair $(\mathcal{G}\times_iP\to M,\omega_\alpha)$ defined above is a Cartan geometry of type $(G,P)$ and any local automorphism of the Cartan geometry $(p: \mathcal{G}\to M, \omega)$ is a local automorphism of the Cartan geometry $(\mathcal{G}\times_iP\to M,\omega_\alpha)$.

Any Cartan geometry of type $(G,P)$ on locally symmetric space $(p: \mathcal{G}\to M, \omega)$ of type $(K,H,h)$ invariant to local symmetries is constructed using an extension of $(K,H)$ to $(G,P)$.
\end{thm}
\begin{proof}
In \cite{parabook} and \cite{greg}, it is shown that the first claim is indeed true. The second claim is there proved globally. This means that generally any Cartan geometry of type $(G,P)$ invariant to local symmetries is only locally isomorphic to an extension of $(K,H)$ to $(G,P)$. But since the group $K$ acts locally transitively and by local symmetries, the second claim also holds. 
\end{proof}

Thus in particular, the group $H$ plays a role for the existence of the extension. However, we get the following corollary of proposition \ref{p1}, because the covering is in fact extension to $(K,H)$ and we can compose extensions.

\begin{cor}
There is extension of $(K,H)$ to $(G,P)$ only if there is extension of the simply connected covering of $(K,H)$ to $(G,P)$.
\end{cor}

\subsection{Holomorphic Cartan geometries}

Finally, we add few remarks about holomorphic Cartan geometries from \cite{parabook}, which are Cartan geometries $(\mathcal{G}\to M,\omega)$ of type $(G,P)$, where $(G,P)$ are complex Lie groups, $\mathcal{G}\to M$ is a holomorphic principal $P$-bundle and $\omega$ is a holomorphic Cartan connection. The curvature of Cartan geometry of type $(G,P)$ splits as a real two form to $\kappa=\kappa_{(2,0)}+\kappa_{(1,1)}+\kappa_{(0,2)}$ according to linearity and anti-linearity with respect to the complex structure. There is the following characterization of holomorphic Cartan geometries of type $(G,P)$:

\begin{prop}\label{4.1.1}
Let $(\mathcal{G}\to M,\omega)$ be Cartan geometry of type $(G,P)$, where $(G,P)$ are complex Lie groups. Then it is holomorphic Cartan geometry of type $(G,P)$ if and only if $\kappa=\kappa_{(2,0)}$ i.e. $\kappa$ is complex bilinear.
\end{prop}

The easy consequence of the formula for the curvature of an extension is:

\begin{cor}\label{comp}
Extension of locally symmetric space of type $(K,H,h)$ to Cartan geometry of type $(G,P)$, where $(G,P)$ are complex Lie groups, is a holomorphic Cartan geometry if and only if $(K,H)$ are complex Lie groups and $\alpha$ is complex linear. If this is the case, there is always a non equivalent extension given by complex conjugation.
\end{cor}

\subsection{AHS--structures}\label{1.2}

Now, we summarize basic results on AHS--structures and one graded parabolic geometries from \cite{parabook}, where one can find more details.

Let $G$ be a semisimple Lie group and $P$ parabolic subgroup corresponding to the grading $\mathfrak{g}=\mathfrak{g}_{-1}\oplus \mathfrak{g}_0\oplus \mathfrak{g}_1$, where we assume that there is no nonzero ideal of $\mathfrak{g}$ in $\mathfrak{g}_0$. Further, we denote $G_0$ the subgroup of grading preserving elements of $P$. The structure of one gradings is the following according to \cite{parabook}:

\begin{prop}\label{4.1.1}
Let $\mathfrak{g}=\mathfrak{g}_{-1}\oplus \mathfrak{g}_0\oplus \mathfrak{g}_1$ be a one graded Lie algebra. Then
\begin{itemize}
\item $\mathfrak{g}$ is a sum of one graded simple Lie algebras $\mathfrak{g}^{(j)}$;
\item the decomposition of $\mathfrak{g}_0$-module $\mathfrak{g}_{-1}$ to irreducible components is given by $\mathfrak{g}_{-1}=\oplus_j \mathfrak{g}_{-1}^{(j)}$;
\item the only non isomorphic one gradings of simple complex and real Lie algebras are those in the table \ref{tab1};
\item  there is equivalence of categories between AHS--structures of type $(G,P)$ and regular normal one graded parabolic geometries of type $(G,P)$.
\end{itemize}
\end{prop}

We use the following symbols and abbreviations in the table \ref{tab1}:

\begin{itemize}
\item[$\mathfrak{g}$] is simple Lie algebra with one grading $\mathfrak{g}=\mathfrak{g}_{-1}+\mathfrak{g}_0+\mathfrak{g}_1$
\item[$\mathfrak{g}_{-1}$] is given as a representation of $ad(\mathfrak{g}_0)$ on $\mathfrak{g}_{-1}$ in terms of fundamental representations.
\end{itemize}

\begin{table}
{
\renewcommand{\arraystretch}{1.2}
\begin{center}
\begin{tabular}{|c|c|c|}
\hline
$\mathfrak{g}$ & $\mathfrak{g}_0$ & $\mathfrak{g}_{-1}$\\
\hline
$\mathfrak{sl}(n,\mathbb{R})$ & $\mathfrak{sl}(p,\mathbb{R})+\mathfrak{sl}(q,\mathbb{R})+\mathbb{R}$ & $\lambda_{p-1}\otimes \lambda_{1}$ \\
\hline
$\mathfrak{sl}(n,\mathbb{C})$ & $\mathfrak{sl}(p,\mathbb{C})+\mathfrak{sl}(q,\mathbb{C})+\mathbb{C}$ & $\lambda_{p-1}\otimes \lambda_{1}$ \\
\hline
$\mathfrak{sl}(n,\mathbb{H})$ & $\mathfrak{sl}(p,\mathbb{H})+\mathfrak{sl}(q,\mathbb{H})+\mathbb{R}$ & $\lambda_{p-1}\otimes \lambda_{1}$ \\
\hline
$\mathfrak{su}(n,n)$ & $\mathfrak{sl}(n,\mathbb{C})+\mathbb{R}$ & $\lambda_{1}+ \bar{\lambda}_{1}$ \\
\hline
$\mathfrak{sp}(2n,\mathbb{R})$ &  $\mathfrak{sl}(n,\mathbb{R})+\mathbb{R}$ & $2\lambda_{1}$  \\
\hline
$\mathfrak{sp}(2n,\mathbb{C})$ &  $\mathfrak{sl}(n,\mathbb{C})+\mathbb{C}$ & $2\lambda_{1}$ \\
\hline
$\mathfrak{sp}(n,n)$ &  $\mathfrak{sl}(n,\mathbb{H})+\mathbb{R}$ & $2\lambda_{1}$  \\
\hline
$\mathfrak{so}(p+1,q+1)$ & $\mathfrak{so}(p,q)+\mathbb{R}$ & $\lambda_{1}$ \\
\hline
$\mathfrak{so}(n+2,\mathbb{C})$ & $\mathfrak{so}(n,\mathbb{C})+\mathbb{C}$ & $\lambda_{1}$ \\
\hline
$\mathfrak{so}(p,p)$ & $\mathfrak{sl}(n,\mathbb{R})+\mathbb{R}$ & $\lambda_{2}$ \\
\hline
$\mathfrak{so}(2n,\mathbb{C})$ & $\mathfrak{sl}(n,\mathbb{C})+\mathbb{C}$ & $\lambda_{2}$ \\
\hline
$\mathfrak{so}^\star(4n)$ & $\mathfrak{sl}(n,\mathbb{H})+\mathbb{R}$ & $\lambda_{2}$ \\
\hline
$\mathfrak{e}^1_6 (EI)$ & $\mathfrak{so}(5,5)+\mathbb{R}$ & $\lambda_{4}$ \\
\hline
$\mathfrak{e}^4_6 (EIV)$ & $\mathfrak{so}(9,1)+\mathbb{R}$ & $\lambda_{4}$ \\
\hline
$\mathfrak{e}^\mathbb{C}_6 (E_6)$ & $\mathfrak{so}(10,\mathbb{C})+\mathbb{C}$ & $\lambda_{4}$\\
\hline
$\mathfrak{e}^1_7 (EV)$ & $\mathfrak{e}^1_6+\mathbb{R}$ & $\lambda_{1}$ \\
\hline
$\mathfrak{e}^3_7 (EVII)$ & $\mathfrak{e}^4_6+\mathbb{R}$ & $\lambda_{1}$\\
\hline
$\mathfrak{e}^\mathbb{C}_7 (E_7)$ & $\mathfrak{e}^\mathbb{C}_6+\mathbb{C}$ & $\lambda_{1}$ \\
\hline
\end{tabular}
\end{center}
}
\caption{Types of AHS--structures}
\label{tab1}
\end{table}

\section{Construction of AHS--structures on locally symmetric spaces}

Now we investigate, how the extensions $(i,\alpha)$ of a symmetric space $(K,H,h)$ to a AHS--structure of type $(G,P)$ can look like up to equivalence. The following lemma characterizes, when two extension leads to locally equivalent geometries, look in \cite{greg} for proof.

\begin{lem}\label{equiv}
Let $(i,\alpha)$ and $(\hat{i},\hat{\alpha})$ be two extension of $(K,H)$ to $(G,P)$. Then the extended geometries are locally equivalent if and only if there are $p_0\in P$ and a Lie algebra automorphism $\sigma$ of $K$ preserving $H$ such, that $\hat{i}(\sigma(h))=p_0i(h)p_0^{-1}$ and $\hat{\alpha}=Ad_{p_0^{-1}}\circ \alpha\circ T\sigma$.
\end{lem}

The first result highly restricts possible AHS--structures on locally symmetric spaces.

\begin{thm}\label{4.1.4}
There is an invariant AHS--structure of type $(G,P)$ on a locally symmetric space of type $(K,H,h)$ if and only if there is an invariant $G_0$-structure on the locally symmetric space.
\end{thm}

The theorem is the direct consequence of the following proposition and Theorem \ref{t1}, because $Ad(G_0)$ is a subgroup of $Gl(\mathfrak{g}_{-1})$.

\begin{prop}
An invariant AHS--structure of type $(G,P)$ on a locally symmetric space $(K,H,h)$ is equivalent to an extension $(i,\alpha)$ of $(K,H)$ to $(G,P)$ such that
\begin{itemize}
\item $i(h)=g_0$;
\item $i(H)\subset G_0$;
\item $\alpha$ is given for $X$ in the $-1$ eigenspace of $Ad(h)$ as follows:

$\alpha(X)_{-1}$ is an arbitrary isomorphism of the adjoint representations \linebreak $Ad(H)$ and $Ad(i(H))$, $\alpha(X)_0=0$, and $\alpha(X)_{1}$ is an arbitrary morphism of the adjoint representations.
\end{itemize}
\end{prop}
\begin{proof}
Let $(i,\alpha)$ be extension of $(K,H,h)$ to AHS--structure of type $(G,P)$, then $i(h)=g_0\exp(Z)$ for $g_0\in G_0$ and $Z\in \mathfrak{g}_1$, and $$e=i(h)^2=g_0\exp(Z)g_0\exp(Z)=g_0^2\exp(Ad(g_0)Z)\exp(Z)=g_0^2\exp(-Z+Z)=g_0^2$$ so we have to assume, that there is $g_0\in G_0,\ g_0^2=e$. Now, we change the extension by conjugation by $\exp(\frac12Z)\in P$ and get equivalent extension according to the above lemma. Then $i(h)=\exp(\frac12Z)g_0\exp(Z)\exp(-\frac12Z)=g_0\exp(Ad(g_0)(\frac12Z))\exp(\frac12Z)=g_0$, thus we can assume that $i(h)=g_0$ without loss of generality.

Since $h$ commutes with $H$, the equality $$g_0p_0\exp(Y)=p_0\exp(Y)g_0$$ has to be satisfied for all $p_0\exp(Y)\in i(H)$. Thus $Y=0$ for all $p_0\exp(Y)\in i(H)$ and $i(H)\subset G_0$.

Now $\alpha(Ad(h)X)=Ad(i(h))\alpha(X)$ for $X$ in the $-1$ eigenspace of $Ad(h)$, thus $$-\alpha(X)=Ad(g_0)\alpha(X).$$  Let us decompose $\alpha(X)=\alpha(X)_{-1}+\alpha(X)_0+\alpha(X)_1$ according to the grading of $\mathfrak{g}$. Then the comparison of both sides provides us restriction $\alpha(X)_{0}=0.$
\end{proof}

Now we solve, when two AHS--structures on a locally symmetric space are locally equivalent.

\begin{thm}
On a locally symmetric space of type $(K,H,h)$, there is a bijection between:
\begin{itemize}
\item equivalence classes of invariant AHS--structure of type $(G,P)$ (up to outer automorphisms of the Lie group $Ad(H)$ induced by automorphisms of $K$).

\item pairs consisting of a conjugacy class of inclusions $i$ of $H$ to $G_0$ and a class of elements of the centralizer of $i(H)$ in $Gl(\mathfrak{g}_{-1})$ contained in $Gl(\mathfrak{g}_{-1})/G_0$.
\end{itemize}
\end{thm}

The theorem is a simple consequence the following technical proposition and Lemma \ref{equiv}.

\begin{prop}\label{p4}
For a symmetric space $(K,H,h)$, there is a bijection between:
\begin{itemize}
\item extensions of $(K,H)$ to AHS--structure of type $(G,P)$ such, that $i(h)=g_0$

\item couples $\beta, b_2$, where $\beta$ is a frame of  the $-1$ eigenspace $\mathfrak{g}_{-1}$ of $Ad(h)$ such, that the inclusion $i_\beta: H\to Gl(\mathfrak{g}_{-1})$ induced by the frame $\beta$ is contained in $G_0$, and $b_2$ is an endomorphism of $\mathfrak{g}_1$ commuting with $i_\beta(Ad(H))$.
\end{itemize}

Two frames of the $-1$ eigenspace of $Ad(h)$ determine the same homomorphism $i: H\to G_0$ if and only if the transition map between them commutes with $i(H)$.

Two frames of the $-1$ eigenspace of $Ad(h)$ determine equivalent AHS--structures of type $(G,P)$ if and only if the transition map between them is composition of elements of $P$ (in fact $G_0$) and outer automorphisms of the Lie group $Ad(H)$ induced by automorphisms of $K$.
\end{prop}
\begin{proof}
The first two claims immediately follows from the previous theorem and lemma, and the same holds for the third claim in all cases except (H--)projective structures, because in these cases the $b_2$-part of the extension plays no role in the question of equivalence. 

In the case of (H--)projective structures, for a fixed frame $\beta$, the $b_2$-part is determined from the normality conditions using formula and notation from \cite{greg} on the curvature $\kappa$:
\[
\begin{split}
0&=(\partial^*\kappa)([e,e])(a_iX^i+\mathfrak{p})=\sum_i [Z_i,[\alpha(a_jX^j),\alpha(X^i)]-\alpha([a_jX^j,X^i])]\\
&=\sum_{i,j} a_j([Z_i,[X^j+b_2(X^j),X^i+b_2(X^i)]-\alpha([X^j,X^i])])\\
&=\sum_{i,j} a_j((b_{ji}-nb_{ij})Z_i-[Z_i,[\alpha([X^j,X^i])])
\end{split}
\]
where $X^i$ is vector in $\mathfrak{g}_{-1}$ with $1$ on i-th row and rest $0$, $Z^i\in \mathfrak{g}_{1}$ is covector with $1$ on i-th column and rest $0$ and $b_2(X^j)=b_{ij}Z_i$. So we get system of linear equations and we know, there always has to be at least one solution. The homogeneous parts of the equations are $(1-n)b_{ii}=0$, $b_{ji}-nb_{ij}=0$ and $b_{ij}-nb_{ji}=0$. Clearly, $b_2=0$ is the only solution of the homogeneous part. Thus there is unique $b_2$ and the last claim follows.
\end{proof}

Since the classification of the semisimple symmetric spaces is known, we can classify all AHS--structures on them. Moreover, a simple consequence of the classification of semisimple locally symmetric spaces is, that the adjoint representation of $H$ on the $-1$ eigenspace $\mathfrak{k}/\mathfrak{h}$ of $Ad(h)$, which is complementary to $\mathfrak{h}$ in $\mathfrak{k}$, is completely reducible. Thus:

\begin{prop}
Let $(K,H,h)$ be semisimple symmetric space. Then $Z_{Gl(\mathfrak{k}/\mathfrak{h})}(Ad(H))$ is product of centralizers $Z_{Gl(\mathfrak{k}_i/\mathfrak{h}_i)}(Ad(H_i))$ of simple factors $(K_i,H_i,h_i)$ of $(K,H,h)$.
\end{prop}

The possible centralizers are described in the case of simple symmetric spaces in \cite{bert}. The only possibilities are $\mathbb{R}$, $\mathbb{C}$, $\mathbb{R}\times \mathbb{R}$ and $\mathbb{C}\times \mathbb{C}$.

Finally, there is a large class of AHS--structures, where the question of equivalence is trivial. This is the result of \cite{zal} summarized in the following proposition.

\begin{prop}
The AHS--structures invariant to local symmetries are torsion-free and if the only component of the harmonic curvature is torsion, then the AHS--structures are unique up to equivalence (because they are locally flat).
\end{prop}

So firstly we investigate the structures, which allow non-trivial curvature and then the rest.

\section{Non-flat invariant AHS--structures}

\subsection{Projective and H--projective structures}

The (H--)projective structures correspond to the following grading of $\mathfrak{g}=\mathfrak{sl}(n+1,\mathbb{K})$ , where $\mathbb{K}=\mathbb{R}$ for projective or $\mathbb{C}$ for H--projective structures:
$$
\left( \begin{array}{cc}
a & Z  \\
X & A  
  \end{array} \right),
$$
where $A\in \mathfrak{gl}(n,\mathbb{K})$, $a=-tr(A)$, $X\in \mathbb{K}^n$ and $Z\in (\mathbb{K}^n)^*$.

The corresponding effective homogeneous model has $G=PGl(n+1,\mathbb{K})$ and
$$
g_0=\left( \begin{array}{cc}
-1 & 0  \\
0 & E  
  \end{array} \right).
$$

Let $(K,H,h)$ be a (complex) symmetric space (see \cite{bert} for details about complex symmetric spaces). Then the choice of any frame $\beta$ of the $-1$ eigenspace of $Ad(h)$ provides $i_\beta: Ad(H)\to Gl(n,\mathbb{K})=G_0$, and it is obvious, that all frames $\beta$ provide equivalent extensions. Thus:

\begin{prop}
There is (up to equivalence) unique projective structure on any semisimple locally symmetric space. There is a H--projective structure on any complex semisimple locally symmetric space. The non-equivalent (up to outer automorphisms of the Lie group $Ad(H)$ induced by automorphisms of $K$) H--projective structures are given by corollary \ref{comp}.
\end{prop}

\subsection{Conformal structures}\label{4.3}

The (complex) conformal structures correspond to the following grading of $\mathfrak{g}=\mathfrak{so}(p+1,q+1)$ or $\mathfrak{g}=\mathfrak{so}(n+2,\mathbb{C})$:
$$
\left( \begin{array}{ccc}
a & Z & 0 \\
X & A & -I_{p,q}Z^T\\
0 & -X^TI_{p,q} & -a
  \end{array} \right),
$$
where $A\in \mathfrak{so}(p,q)$, $a\in \mathbb{R}$, $X\in \mathbb{R}^{p+q}$, $Z\in (\mathbb{R}^{p+q})^*$ and is diagonal matrix $I_{p,q}$ with $\pm 1$ on diagonal according to the signature $(p,q)$, or $A\in \mathfrak{so}(n,\mathbb{C})$, $a\in \mathbb{C}$, $X\in \mathbb{C}^{n}$, $Z\in (\mathbb{C}^{n})^*$ and $I_{p,q}$ is identity matrix.

The corresponding effective model has $G=PO(p+1,q+1)$ or $G=PO(n+2,\mathbb{C})$ and
$$
g_0=\left( \begin{array}{ccc}
-1 & 0 & 0 \\
0 & E & 0\\
0 & 0 & -1
  \end{array} \right).
$$

On a semisimple symmetric space $(K,H,h)$, we can restrict the Killing form $B: \mathfrak{k}\otimes \mathfrak{k}\to \mathbb{R}$ to the $H$-invariant complement $\mathfrak{k}/\mathfrak{h}$ of $\mathfrak{h}$. This defines a non-degenerate $Ad(H)$-invariant symmetric bilinear form on $T_eK/H$, which defines an inclusion $i:H\to O(n-p,p)$, where $p=\operatorname{dim}(C)-\operatorname{dim}(C\cap H)$ and $C$ is the maximal compact subgroup of $K$.

Moreover, the Killing form $B$ provides an $Ad(H)$-invariant bijection $$\mathfrak{gl}(\mathfrak{k}/\mathfrak{h})\to (\mathfrak{k}/\mathfrak{h}\times \mathfrak{k}/\mathfrak{h})^*,\ X\mapsto B(X \cdot, \cdot).$$ 
In particular, $B(X \cdot, \cdot)$ is a non-degenerate $Ad(H)$-invariant symmetric bilinear form on $T_eK/H$ if and only if $X$ is an element of the centralizer of $i(H)$ in $Gl(\mathfrak{g}_{-1})$ contained in $Gl(\mathfrak{g}_{-1})/O(n-p,p)$. Thus:

\begin{prop}
For any semisimple locally symmetric space of type $(K,H,h)$ holds:
\begin{itemize}
\item The simple factors are orthogonal to each other with respect to any invariant metric and the only non-degenerate invariant symmetric bilinear form on a simple factor is a real multiple (a complex multiple if $B$ is complex linear) of the Killing form $B$.
\item There is unique (up to conjugacy) Lie group homomorphism $i:H\to G_0=CO(p,q)$ if and only if there are multiples of the Killing forms of the simple factors such that the resulting form has signature $(p,q)$. There is bijection between the conjugacy classes of conformal structures on the locally symmetric space and $(\mathbb{S}^1)^r$, where $r$ is the number of the simple factors with complex linear Killing form.
\item There is a complex conformal structure if and only if there is a complex structure on $(K,H,h)$. The non-equivalent (up to outer automorphisms of the Lie group $Ad(H)$ induced by automorphisms of $K$) complex conformal structures are given by corollary \ref{comp}.
\end{itemize}
\end{prop}
\begin{proof}
The first claim clearly follows from description of elements of the centralizer of $i(H)$ in $Gl(\mathfrak{g}_{-1})$ contained in $Gl(\mathfrak{g}_{-1})/O(n-p,p)$. It follows from \cite{bert}, that only those in claim are possible. Second claim clearly follows, because we are taking $Gl(\mathfrak{g}_{-1})/CO(p,q)$ instead of $Gl(\mathfrak{g}_{-1})/O(p,q)$. The third claim then clearly follows from Lemma \ref{comp}. 
\end{proof}

\subsection{Quaternionic structures}\label{4.4}

The quaternionic structures correspond to the following grading of $\mathfrak{g}=\mathfrak{sl}(n+1,\mathbb{H})$:
$$
\left( \begin{array}{cc}
a & Z \\
X & A 
  \end{array} \right),
$$
where $A\in \mathfrak{gl}(n,\mathbb{H})$, $a\in \mathbb{H}$, $\operatorname{Re}(a)+\operatorname{Re}(tr(A))=0$, $X\in \mathbb{H}^{p+q}$ and $Z\in (\mathbb{H}^{p+q})^*$.

The corresponding effective model has $G=PGl(n+1,\mathbb{H})$ and
$$
g_0=\left( \begin{array}{cc}
-1 & 0 \\
0 & E
  \end{array} \right).
$$

We will assume that $n>1$, because $\mathfrak{sl}(2,\mathbb{H})\cong \mathfrak{so}(5,1)$ and the parabolic geometry is in fact conformal geometry, which we already discussed.

\begin{exam}
Quaternionic structure on $(\mathfrak{so}^*(2n+2), \mathfrak{so}^*(2)\oplus \mathfrak{so}^*(2n))$.

If we look in classification of simple symmetric spaces, $SO^*(2n)$ acts by a quaternionic representation i.e. there is $i: SO^*(2n)\to Gl(n,\mathbb{H})$. Further $SO^*(2)$ acts by multiples of $-k\in Sp(1)$ from left i.e. $SO^*(2)\times SO^*(2n)$ sits uniquely up to conjugation in $G_0:=P(Sp(1)\times Gl(n,\mathbb{H}))$. In fact, we immediately get a flat quaternionic structure on $SO^*(2n+2)/ SO^*(2)\times SO^*(2n)$ just by inclusion of $SO^*(2n+2)$ to $PGl(n+1,\mathbb{H})$.
\end{exam}

We show that there is a quaternionic structure on any pseudo--quater\-nionic--K\"ahler symmetric space and there are no quaternionic structures on other semisimple symmetric spaces except the previous example.

\begin{prop}
There is a quaternionic structure on any pseudo--quaternionic--K\"ahler locally symmetric space and there are no quaternionic structures on other semisimple locally symmetric spaces except $(\mathfrak{so}^*(2n+2), \mathfrak{so}^*(2)\oplus \mathfrak{so}^*(2n))$. The structure is unique up to equivalence.
\end{prop}
\begin{proof}
Let $(K,H,h)$ be a semisimple homogeneous symmetric space and assume that the image of $i$ is contained in $Gl(n,\mathbb{H})$. Then the representation of $Ad(i(H))$ is of quternionic type and there is no such in the classification of semisimple symmetric spaces. The same is true in the case that the image of $i$ is contained in the part given by $a\in \mathbb{H}$. So the image of $i$ has intersection with both parts, but this implies that the representation of $Ad(i(H))$ is irreducible and going through the list of simple symmetric spaces we check that the only possibilities are pseudo-quaternionic-K\"ahler symmetric spaces (where $Sp(1)\times Sp(p,q)$ trivially sits in $G_0$) and the previous example. Since $i(H)$ acts by irreducible representations, its image is unique, and the centralizer of $i(H)$ in $Gl(\mathfrak{g}_{-1})$ contained $Gl(\mathfrak{g}_{-1})/G_0$ is trivial.
\end{proof}

\subsection{Para-quaternionic structures}\label{4.5}

The para-quaternionic structures correspond to the following grading of $\mathfrak{g}=\mathfrak{sl}(n+2,\mathbb{R})$:
$$
\left( \begin{array}{cc}
a & Z \\
X & A 
  \end{array} \right),
$$
where $A\in \mathfrak{gl}(n,\mathbb{R})$, $a\in \mathfrak{gl}(2,\mathbb{R})$, $tr(a)+tr(A)=0$, $X\in \mathbb{R}^{n}\otimes (\mathbb{R}^{2})^*$ and $Z\in \mathbb{R}^{2}\otimes (\mathbb{R}^{n})^*$.

The corresponding effective model has $G=PGl(n+2,\mathbb{R})$ and
$$
g_0=\left( \begin{array}{ccc}
-1 & 0& 0 \\
0 & -1& 0 \\
0 & 0 & E
  \end{array} \right).
$$

We will assume that $n>2$, because $\mathfrak{sl}(4,\mathbb{R})\cong \mathfrak{so}(3,3)$ and the parabolic geometry is in fact conformal geometry, which we already discussed.

\begin{exam}
Para-quaternionic structures on $(\mathfrak{so}(k+1,l+1),\mathfrak{so}(1,1)\oplus \mathfrak{so}(k,l))$ and  $(\mathfrak{so}(k+2,l),\mathfrak{so}(2)\oplus \mathfrak{so}(k,l))$.

First we notice that $SO(n+2)/SO(2)\times SO(n)$ is equivalent to the homogeneous model of $(G,P)$. Thus $SO(2)\times SO(n)$ sits uniquely up to conjugation in $G_0:=P(Gl(2,\mathbb{R})\times Gl(n,\mathbb{R}))$. Since it does not matter, which signature the matrices have, we immediately get a flat para-quaternionic structure on $SO(k+1,l+1)/SO(1,1)\times SO(k,l)$ and  $SO(k,l+2)/SO(2)\times SO(k,l)$ just by inclusion.
\end{exam}

We show that there is a para-quaternionic structure on any pseudo-para-quaternionic-K\"ahler symmetric space and there are no para-quaternionic structures on other  semisimple symmetric spaces except the previous examples.

\begin{prop}
There is a para-quaternionic structure on any pseudo-para-quaternionic-K\"ahler locally symmetric space and there are no para-quaternionic structures on other semisimple locally symmetric spaces except $(\mathfrak{so}(k+1,l+1),\mathfrak{so}(1,1)\oplus \mathfrak{so}(k,l))$ and  $(\mathfrak{so}(k+2,l),\mathfrak{so}(2)\oplus \mathfrak{so}(k,l))$. The structure is unique up to equivalence.
\end{prop}
\begin{proof}
Let $(K,H,h)$ be a semisimple homogeneous symmetric space and assume that the image of $i$ is contained in $Sl(n,\mathbb{R})$. Then the representation of $Ad(i(H))$ decomposes to two copies of standard representation of $\mathfrak{sl}(n,\mathbb{R})$ and there is no such in the classification of semisimple symmetric spaces. The same is true in the case that the image of $i$ is contained in the part given by $a\in \mathfrak{gl}(2,\mathbb{R})$. So the image of $i$ has intersection with both parts, but this implies that the representation of $Ad(i(H))$ is irreducible and going through the list of simple symmetric spaces we check that the only possibilities are pseudo-para-quaternionic-K\"ahler symmetric spaces (where $Sp(2,\mathbb{R})\times Sp(2n,\mathbb{R})$ trivially sits in $P(Gl(2,\mathbb{R})\times Gl(2n,\mathbb{R}))$) and those in previous example.  Since $i(H)$ acts by irreducible representations, its image is unique, and the centralizer of $i(H)$ in $Gl(\mathfrak{g}_{-1})$ contained $Gl(\mathfrak{g}_{-1})/G_0$ is trivial.
\end{proof}

\section{Flat invariant AHS--structures}

Many of these symmetric spaces and structures on them are described in \cite{bert} in language of algebraic geometry. We present complete list of such structures.

\subsection{General Grassmannian structures}

The Grassmannian structures of type $(p,q)$ correspond to the following grading of $\mathfrak{g}=\mathfrak{sl}(p+q,\mathbb{K})$, where $p>2, q>2$ and $\mathbb{K}=\mathbb{R}$, or $p>1, q>1$ and $\mathbb{K}=\mathbb{C}$, or $p>1, q>1$ and $\mathbb{K}=\mathbb{H}$:
$$
\left( \begin{array}{cc}
A & Z \\
X & B 
  \end{array} \right),
$$
where $A\in \mathfrak{gl}(p,\mathbb{K})$, $B\in \mathfrak{gl}(q,\mathbb{K})$, $\operatorname{Re}(tr(a))+\operatorname{Re}(tr(A))=0$, $X\in \mathbb{K}^{q}\otimes (\mathbb{K}^{p})^*$ and $Z\in \mathbb{K}^{q}\otimes (\mathbb{K}^{p})^*$.

The corresponding effective model has $G=PGl(p+q,\mathbb{K})$ and
$$
g_0=\left( \begin{array}{cc}
-E_p & 0 \\
0 & E_q
  \end{array} \right),
$$
where $E_p$, $E_q$ are identity matrices of sizes $p\times p$ and $q\times q$.

We show that general Grassmannian structures are only on simple symmetric spaces and list them.

\begin{prop}
The only locally symmetric spaces admitting a (unique) general Grassmannian structure are the following types of simple symmetric spaces:

 $(\mathfrak{so}(a+b,c+d),\mathfrak{so}(a,c)\oplus \mathfrak{so}(b,d))$, $\mathbb{K}=\mathbb{R}$ of type $(a+c,b+d)$; 

$(\mathfrak{sp}(2(p+q),\mathbb{R}),\mathfrak{sp}(2p,\mathbb{R})\oplus \mathfrak{sp}(2q,\mathbb{R}))$, $\mathbb{K}=\mathbb{R}$ of type $(2p,2q)$;

$(\mathfrak{so}(p+q,\mathbb{C}),\mathfrak{so}(p,\mathbb{C})\oplus \mathfrak{so}(q,\mathbb{C}))$, $\mathbb{K}=\mathbb{C}$ of type $(p,q)$, holomorphic Cartan geometry;

$(\mathfrak{sp}(p+q,\mathbb{C}),\mathfrak{sp}(p,\mathbb{C})\oplus \mathfrak{sp}(q,\mathbb{C}))$, $\mathbb{K}=\mathbb{C}$ of type $(p,q)$, holomorphic Cartan geometry;

$(\mathfrak{so}^*(2(p+q)),\mathfrak{so}^*(2p)\oplus \mathfrak{so}^*(2q))$, $\mathbb{K}=\mathbb{H}$ of type $(2p,2q)$;

$(\mathfrak{sp}(a+b,c+d),\mathfrak{sp}(a,c)\oplus \mathfrak{sp}(b,d))$, $\mathbb{K}=\mathbb{H}$ of type $(a+c,b+d)$;

$(\mathfrak{su}(a+b,c+d),\mathfrak{su}(a,c)\oplus \mathfrak{su}(b,d)) \oplus \mathfrak{so}(2)$, $\mathbb{K}=\mathbb{C}$ of type $(a+c,b+d)$.
\end{prop}
\begin{proof}
Let $(K,H,h)$ be a semisimple homogeneous symmetric space and assume that the image of $i$ is contained in $Sl(p,\mathbb{K})$ or $Sl(q,\mathbb{K})$. Then the representation of $Ad(i(H))$ decomposes to several copies of standard representation of $\mathfrak{sl}(p,\mathbb{R})$ and there is no such in the classification of semisimple symmetric spaces. So the image of $i$ has intersection with both parts, but this implies that the representation of $Ad(i(H))$ is irreducible and going through the list of simple symmetric spaces we see that the only those in the proposition act by the prescribed representations of $G_0$. 
\end{proof}

We note that there the following types of symmetric spaces with Grassmannian structures, which are not semisimple. The inclusion to $G_0$ is given by the adjoint representation.

$\mathbb{R}\oplus \mathfrak{sl}(n,\mathbb{R})$, $\mathbb{K}=\mathbb{R}$ of type $(n,n)$;

$\mathbb{C}\oplus \mathfrak{sl}(n,\mathbb{C})$, $\mathbb{K}=\mathbb{C}$ of type $(n,n)$, holomorphic Cartan geometry;

$\mathbb{R}\oplus \mathfrak{sl}(n,\mathbb{H})$, $\mathbb{K}=\mathbb{H}$ of type $(n,n)$.

\subsection{General Lagrangean structures}

The Lagrangean structures correspond to the following grading of $\mathfrak{g}=\mathfrak{sp}(2n,\mathbb{K})$, where $\mathbb{K}=\mathbb{R}$, or $\mathbb{K}=\mathbb{C}$, or $\mathfrak{g}=\mathfrak{sp}(n,n)$ in the case $\mathbb{K}=\mathbb{H}$:
$$
\left( \begin{array}{cc}
-A^T & Z \\
X &  A
  \end{array} \right),
$$
where $A\in \mathfrak{gl}(n,\mathbb{K})$, , $X\in S^2 \mathbb{K}^{n}$ and $Z\in (S^2 \mathbb{K}^{n})^*$.

The corresponding effective model has $G=Sp(2n,\mathbb{K})\rtimes \mathbb{Z}_2$ or $Sp(n,n)\rtimes \mathbb{Z}_2$, because
$$
g_0=\left( \begin{array}{cc}
-E_n & 0 \\
0 & E_n
  \end{array} \right),
$$
where $E_n$ is $n\times n$ identity matrix, is not in the connected component of identity of $G$.

We show that general Lagrangean structures are only on simple symmetric spaces and list them.

\begin{prop}
The only locally symmetric spaces admitting a (unique) general Lagrangean structure are the following types of simple symmetric spaces:

$(\mathfrak{sp}(2n,\mathbb{C}),\mathfrak{sp}(2n,\mathbb{R}))$, $\mathbb{K}=\mathbb{R}$ ;

$\mathfrak{sp}(2n,\mathbb{R})$, $\mathbb{K}=\mathbb{R}$;

$\mathfrak{sp}(n,\mathbb{C})$, $\mathbb{K}=\mathbb{C}$, holomorphic Cartan geometry;

$(\mathfrak{sp}(n,\mathbb{C}),\mathfrak{sp}(p,q))$, $\mathbb{K}=\mathbb{H}$;

$\mathfrak{sp}(p,q)$, $\mathbb{K}=\mathbb{H}$;

$(\mathfrak{sp}(2n,\mathbb{R}),\mathfrak{su}(p, q) \oplus \mathfrak{so}(2))$, $\mathbb{K}=\mathbb{C}$;

$(\mathfrak{sp}(p,q),\mathfrak{su}(p, q) \oplus \mathfrak{so}(2))$, $\mathbb{K}=\mathbb{C}$.
\end{prop}
\begin{proof}
Let $(K,H,h)$ be a semisimple homogeneous symmetric space and assume $i(H)\subset G_0=Gl(n,\mathbb{K})$. Since the representation of $i(H)$ is completely reducible, there is invariant complement to any invariant subspace simple factor of $i(H)$ in $S^2 \mathbb{K}^{n}$. Since due to structure of $S^2 \mathbb{K}^{n}$ there is at most one invariant subspace with non-trivial action of $i(H)$, the symmetric space has to be simple.  Then going through the list of simple symmetric spaces we see that the only those in the proposition act by the prescribed representations of $G_0$ except last two examples. The last two examples corresponds to the maximal compact subgroup of $\mathfrak{sp}(n,\mathbb{C})$ and the signature and the real form of $\mathfrak{k}$ plays no role.
\end{proof}

Of course, if we add abelian factors (which are mapped to diagonal in $S^2 \mathbb{K}^{n}$)  to certain types of simple symmetric spaces we obtain the following symmetric spaces with general Lagrangean structures:

$(\mathbb{R}\oplus \mathfrak{su}(p,q),\mathfrak{so}(p,q))$, $\mathbb{K}=\mathbb{R}$;

$(\mathbb{R}\oplus \mathfrak{sl}(p+q,\mathbb{R}),\mathfrak{so}(p,q))$, $\mathbb{K}=\mathbb{R}$;

$(\mathbb{C}\oplus \mathfrak{sl}(n,\mathbb{C}),\mathfrak{so}(n,\mathbb{C}))$, $\mathbb{K}=\mathbb{C}$, holomorphic Cartan geometry;

$(\mathbb{R}\oplus \mathfrak{su}(n,n),\mathfrak{so}^*(2n))$, $\mathbb{K}=\mathbb{H}$;

$(\mathbb{R}\oplus \mathfrak{sl}(n,\mathbb{H}),\mathfrak{so}^*(2n))$, $\mathbb{K}=\mathbb{H}$.

\subsection{General spinorial structures}

The spinorial structures correspond to the following grading of $\mathfrak{g}=\mathfrak{so}(n,n)$ for $\mathbb{K}=\mathbb{R}$, $\mathfrak{g}=\mathfrak{so}(2n,\mathbb{C})$ for $\mathbb{K}=\mathbb{C}$, or $\mathfrak{g}=\mathfrak{so}^\star(4n)$ for $\mathbb{K}=\mathbb{H}$:
$$
\left( \begin{array}{cc}
-A^T & Z \\
X &  A
  \end{array} \right),
$$
where $A\in \mathfrak{gl}(n,\mathbb{K})$, , $X\in \bigwedge^2 \mathbb{K}^{n}$ and $Z\in (\bigwedge^2 \mathbb{K}^{n})^*$.

The corresponding effective model has $G=SO(n,n)\rtimes \mathbb{Z}_2$, $SO(2n,\mathbb{C})\rtimes \mathbb{Z}_2$ or $SO^*(4n)\rtimes \mathbb{Z}_2$, because
$$
g_0=\left( \begin{array}{cc}
-E_n & 0 \\
0 & E_n
  \end{array} \right),
$$
where $E_n$ is $n\times n$ identity matrix, is not in the connected component of identity of $G$.

We show that general spinorial structures are only on simple symmetric spaces and list them.

\begin{prop}
The only locally symmetric spaces admitting a (unique) general spinorial structure are the following types of simple symmetric spaces:

$\mathfrak{so}(p,q)$, $\mathbb{K}=\mathbb{R}$;

$(\mathfrak{so}(n,\mathbb{C}),\mathfrak{so}(p,q))$, $\mathbb{K}=\mathbb{R}$;

$\mathfrak{so}(n,\mathbb{C})$, $\mathbb{K}=\mathbb{C}$, holomorphic Cartan geometry;

$(\mathfrak{so}(2n,\mathbb{C}),\mathfrak{so}^*(2n))$, $\mathbb{K}=\mathbb{H}$;

$\mathfrak{so}^*(2n)$, $\mathbb{K}=\mathbb{H}$;

$(\mathfrak{so}(2p,2q),\mathfrak{su}(p, q)\oplus \mathfrak{so}(2))$, $\mathbb{K}=\mathbb{C}$;

$(\mathfrak{so}^*(2n),\mathfrak{su}(p,q) \oplus \mathfrak{so}(2))$, $\mathbb{K}=\mathbb{C}$.
\end{prop}
\begin{proof}
Let $(K,H,h)$ be a semisimple homogeneous symmetric space and assume $i(H)\subset G_0=Gl(n,\mathbb{K})$. Since the representation of $i(H)$ is completely reducible, there is invariant complement to any invariant subspace simple factor of $i(H)$ in $\bigwedge^2 \mathbb{K}^{n}$. Since due to structure of $\bigwedge^2 \mathbb{K}^{n}$ there is at most one invariant subspace with non-trivial action of $i(H)$, the symmetric space has to be simple.  Then going through the list of simple symmetric spaces we see that the only those in the proposition act by the prescribed representations of $G_0$ except last two examples. The last two examples corresponds to the maximal compact subgroup of $\mathfrak{so}(2n,\mathbb{C})$ and the signature and the real form of $\mathfrak{k}$ plays no role.
\end{proof}

Of course, if we add abelian factors to certain types of simple symmetric spaces we obtain the following symmetric spaces with general spinorial structures:

$(\mathfrak{u}(n,n),\mathfrak{sp}(2n,\mathbb{R}))$, $\mathbb{K}=\mathbb{R}$;

$(\mathfrak{gl}(n,\mathbb{R}),\mathfrak{sp}(2n,\mathbb{R}))$, $\mathbb{K}=\mathbb{R}$;

$(\mathfrak{gl}(2n,\mathbb{C}),\mathfrak{sp}(2n,\mathbb{C}))$, $\mathbb{K}=\mathbb{C}$, holomorphic Cartan geometry;

$(\mathfrak{gl}(n,\mathbb{H}),\mathfrak{sp}(p,q))$, $\mathbb{K}=\mathbb{H}$;

$(\mathfrak{u}(2p,2q),\mathfrak{sp}(p,q))$, $\mathbb{K}=\mathbb{H}$.

\subsection{$\mathfrak{su}(p,p)$-Cartan geometries}

The $\mathfrak{su}(p,p)$-Cartan geometries structures correspond to the following grading of $\mathfrak{g}=\mathfrak{su}(p,p)$:
$$
\left( \begin{array}{cc}
-A^T & Z \\
X &  A
  \end{array} \right),
$$
where $A\in \mathfrak{sl}(n,\mathbb{C})$, and $X, Z\in  \mathfrak{u}(n)$.

The corresponding effective model has $G=SU(n,n)\rtimes \mathbb{Z}_2$, because
$$
g_0=\left( \begin{array}{cc}
-E_n & 0 \\
0 & E_n
  \end{array} \right),
$$
where $E_n$ is $n\times n$ identity matrix, is not in the connected component of identity of $G$.

We show that $\mathfrak{su}(p,p)$-Cartan geometries are only on simple symmetric spaces and list them.

\begin{prop}
The only locally symmetric spaces admitting a (unique) $\mathfrak{su}(p,p)$-Cartan geometry are the following types of simple symmetric spaces:

$(\mathfrak{so}(n,n),\mathfrak{so}(n,\mathbb{C}))$; $(\mathfrak{so}^*(2n),\mathfrak{so}(n,\mathbb{C}))$; $(\mathfrak{sp}(n,n),\mathfrak{sp}(n,\mathbb{C}))$; $(\mathfrak{sp}(2n,\mathbb{R}),\mathfrak{sp}(n,\mathbb{C}))$, $(\mathfrak{sl}(2n,\mathbb{R}),\mathfrak{sl}(n,\mathbb{C})\oplus \mathfrak{so}(2))$; $(\mathfrak{sl}(n,\mathbb{H})/\mathfrak{sl}(n,\mathbb{C})\oplus \mathfrak{so}(2))$.
\end{prop}
\begin{proof}
Let $(K,H,h)$ be a semisimple homogeneous symmetric space and assume $i(H)\subset G_0=S(Gl(n,\mathbb{C}))$. Since the representation of $i(H)$ is completely reducible, there is invariant complement to any invariant subspace simple factor of $i(H)$ in $ \mathfrak{u}(n)$. Since due to structure of $ \mathfrak{u}(n)$ there is at most one invariant subspace with non-trivial action of $i(H)$, the symmetric space has to be simple.  Then going through the list of simple symmetric spaces we see that the only those in the proposition act by the prescribed representations of $G_0$ except last two examples. For the last two examples $\mathfrak{sl}(n,\mathbb{C})$ is mapped to real part of $G_0=\mathfrak{sl}(2n,\mathbb{C})$ and $\mathfrak{so}(2)$ is mapped to complex diagonal.
\end{proof}

Of course, if there is a $\mathfrak{su}(p,p)$-Cartan geometry on $\mathfrak{u}(p,q)$, which is not semisimple.

\subsection{Exceptional structures}

There are six AHS--structures corresponding to exceptional Lie groups. We discuss them only briefly. 

For the first three $G_0=SO(5,5)$, $G_0=SO(9,1)$ or $G_0=SO(10,\mathbb{C})$ and acts by spin representation. Consequently such structures can exist only on simple symmetric spaces and going trough the list of simple symmetric spaces we obtain the following lists:

For $G_0=SO(5,5)$:

$(\mathfrak{f}^1_4,\mathfrak{so}(5,4))$; $(\mathfrak{sp}(4,\mathbb{R}),\mathfrak{sp}(2,\mathbb{R})\oplus \mathfrak{sp}(2,\mathbb{R}))$; $(\mathfrak{sp}(2,2),\mathfrak{sp}(1,1)\oplus \mathfrak{sp}(1,1))$; $(\mathfrak{sp}(4),\mathfrak{sp}(2)\oplus \mathfrak{sp}(2))$; $(\mathfrak{sp}(2,2),\mathfrak{sp}(2)\oplus \mathfrak{sp}(2))$;

For $G_0=SO(9,1)$:

$(\mathfrak{f}^2_4,\mathfrak{so}(9))$; $(\mathfrak{f}^2_4,\mathfrak{so}(8,1))$; $(\mathfrak{f}_4,\mathfrak{so}(9))$; $(\mathfrak{sp}(3,1),\mathfrak{sp}(1,1)\oplus \mathfrak{sp}(2))$

For $G_0=SO(10,\mathbb{C})$:

$(\mathfrak{f}_4^\mathbb{C},\mathfrak{so}(9,\mathbb{C}))$, holomorphic Cartan geometry; $(\mathfrak{sp}(4,\mathbb{C}),\mathfrak{sp}(2,\mathbb{C})\oplus \mathfrak{sp}(2,\mathbb{C}))$, holomorphic Cartan geometry; $(\mathfrak{e}_6,\mathfrak{so}(10)\oplus \mathfrak{so}(2))$; $(\mathfrak{e}^3_6,\mathfrak{so}(10)\oplus \mathfrak{so}(2))$;

The inclusions $i:H\to G_0$ are in above cases induced by inclusions of maximal compact subgroups and we use the isomorphism $\mathfrak{so}(5,\mathbb{C})\cong \mathfrak{sp}(2,\mathbb{C})$ (and corresponding isomorphism of real forms).

For the second three $G_0=E_6^1$, $G_0=E_6^4$ or $G_0=E_6^\mathbb{C}$ and acts by standard representation. Consequently such structures can exist only on simple symmetric spaces and going trough the list of simple symmetric spaces we obtain the following lists, where we also add symmetric spaces with these structures coming from simple symmetric spaces after adding abelian factor):

For $G_0=E_6^1$:

$(\mathfrak{su}(8),\mathfrak{sp}(4))$; $(\mathfrak{su}(4,4),\mathfrak{sp}(2,2))$; $(\mathfrak{su}(4,4),\mathfrak{sp}(4,\mathbb{R}))$; $(\mathfrak{sl}(4,\mathbb{H}),\mathfrak{sp}(4))$; $(\mathfrak{sl}(4,\mathbb{H}),\mathfrak{sp}(2,2))$; $(\mathbb{R}\oplus \mathfrak{e}^1_6,\mathfrak{f}^1_4)$; $(\mathbb{R}\oplus \mathfrak{e}^2_6,\mathfrak{f}^1_4)$; 

For $G_0=E_6^4$:

$(\mathfrak{su}(6,2),\mathfrak{sp}(3,1))$; $(\mathfrak{sl}(4,\mathbb{H}),\mathfrak{sp}(3,1))$;  $(\mathbb{R}\oplus \mathfrak{e}_6,\mathfrak{f}_4)$;  $(\mathbb{R}\oplus \mathfrak{e}^4_6,\mathfrak{f}_4)$;  $(\mathbb{R}\oplus \mathfrak{e}^3_6,\mathfrak{f}^2_4)$;  $(\mathbb{R}\oplus \mathfrak{e}^4_6,\mathfrak{f}^2_4)$; 

For $G_0=E_6^\mathbb{C}$:

$(\mathfrak{sl}(8,\mathbb{C}),\mathfrak{sp}(4,\mathbb{C})$, holomorphic Cartan geometry; $(\mathbb{C}\oplus \mathfrak{e}_6^\mathbb{C},\mathfrak{f}_4^\mathbb{C})$, holomorphic Cartan geometry; $(\mathfrak{e}_7,\mathfrak{e}_6\times \mathfrak{so}(2))$; $(\mathfrak{e}^3_7/\mathfrak{e}_6\times \mathfrak{so}(2))$;

The inclusions $i:H\to G_0$ are in above cases induced by inclusions of maximal compact subgroups.

\section*{Acknowledgements}

The author has been supported by the grant GACR 201/09/H012.

\end{document}